\documentclass{amsart}
\usepackage[dvipsnames]{xcolor}
\pdfoutput=1

\usepackage{amssymb, amsmath}
\usepackage{tikz}
\usepackage{tikz-cd}
\usetikzlibrary{arrows}
\usetikzlibrary{fit}
\usepackage{bbm} 

\usepackage{float}

\RequirePackage{color}
\definecolor{myred}{rgb}{0.75,0,0}
\definecolor{mygreen}{rgb}{0,0.5,0}
\definecolor{myblue}{rgb}{0,0,0.65}

\RequirePackage{ifpdf}
\ifpdf
  \IfFileExists{pdfsync.sty}{\RequirePackage{pdfsync}}{}
  \RequirePackage[pdftex,
   colorlinks = true,
   urlcolor = myblue, 
   citecolor = mygreen, 
   linkcolor = myred, 
   bookmarksopen=true]{hyperref}
\else
  \RequirePackage[hypertex]{hyperref}
\fi

\RequirePackage{ae, aecompl, aeguill} 

\DeclareFontFamily{U}{mathx}{}
\DeclareFontShape{U}{mathx}{m}{n}{<-> mathx10}{}
\DeclareSymbolFont{mathx}{U}{mathx}{m}{n}
\DeclareMathAccent{\widehat}{0}{mathx}{"70}
\DeclareMathAccent{\widecheck}{0}{mathx}{"71}

\newcommand{\co}{\colon}

\newtheorem{thm}{Theorem}[section]
\newtheorem{lem}[thm]{Lemma}
\newtheorem{prop}[thm]{Proposition}

\theoremstyle{definition}
\newtheorem{defn}[thm]{Definition}

\newtheorem{ex}[thm]{Example}

\theoremstyle{remark}
\newtheorem{rem}[thm]{Remark}

\def\ZZ{\mathbb{Z}}

\DeclareMathOperator{\id}{id}

\DeclareMathOperator{\RD}{RD}

\newcommand{\ufd}{U}
\newcommand{\lng}[1]{w_0^{#1}}
\newcommand{\lngi}[1]{w_0^{#1,i}}

\usepackage[maxbibnames=99,
			maxalphanames=10, 
			style=alphabetic]{biblatex}
\usepackage{xurl}

\title[]{Coxeter embeddings are injective}

\author[]{Ben Elias}
\address{University of Oregon.}
\email{belias@uoregon.edu}

\author[]{Edmund Heng}
\address{Institut des Hautes \'Etudes Scientifiques (IHES).}
\email{heng@ihes.fr}

\begin{document}

\begin{abstract} 
We show that certain embeddings of Coxeter groups within other Coxeter groups are injective using the notion of Coxeter partitions.
Moreover, we study Lusztig's partitions, which are generalizations of Lusztig's admissible maps and Crisp's foldings.
We show that they classify the simplest type of Coxeter partitions, whose embeddings of Coxeter groups send each generator to a product of commuting generators.
Consequently, these embeddings are also injective, and we prove that they preserve Coxeter numbers. These results were previously known, due to work of M\"{u}hlherr and Dyer.
\end{abstract}

\maketitle

\section*{Authors' notes}
This short paper contains no novel results. When we first wrote this paper we believed Theorems \ref{strongermainthm} and \ref{prop:thmfordihedral} were new, after having asked experts and failing to find it in the literature. Subsequently, thanks to some helpful leads, we eventually tracked down the original sources. We will leave this paper on the arXiv to make the results more accessible, but do not intend to publish it. Discussion of prior work can be found in \S\ref{sec:relatedworks}.

\section{Coxeter partitions and Coxeter embeddings}
Throughout, a Coxeter system will be denoted as $(W,S)$, where $W$ is a Coxeter group and $S$ is the corresponding set of Coxeter generators.
For $J \subseteq S$, the corresponding parabolic subgroup of $W$ will be denoted by $W_J$.
Given words $u$ and $w$, and $m \geq 2$, the notation $\langle u, w \rangle^m$ denotes the $m$-alternating expression:
\[
\langle u, w \rangle^m:= \underbrace{uwu\ldots}_{m \text{ times}}.
\]
A subword is always assumed to be contiguous.

A surjective function $\pi \co \hat{S} \to S$ will partition the set $\hat{S}$ into a disjoint union of preimages $\ufd(s) := \pi^{-1}(s)$.

\begin{defn}
Let $(\hat{W}, \hat{S})$ be a Coxeter system, where $\hat{S}$ is finite, and let $S$ be a finite set. 
A \emph{Coxeter partition} is a surjective function $\pi \co \hat{S} \to S$ such that 
\begin{enumerate}
	\item For each $s \in S$, the parabolic subgroup $\hat{W}_{\ufd(s)}$ associated to $\ufd(s) \subseteq \hat{S}$ is a finite Coxeter group. Let $\lng{s}$ denote the longest element of $\hat{W}_{\ufd(s)}$.
	\item For each pair of elements $s \neq t \in S$, we denote the order of the element $\lng{s}\lng{t} \in \hat{W}$ by $m_{st}$ (possibly infinite). For some (and hence every) reduced expression of $\lng{s}$ and $\lng{t}$, we require the following: 
		\begin{enumerate}
		\item If $m_{st} < \infty$, then the parabolic subgroup associated to $\ufd(s) \sqcup \ufd(t)$ is finite, and $\langle \lng{s},\lng{t} \rangle^{m_{st}}$ is a reduced expression of its longest element.
		\item If $m_{st} = \infty$, then for all $k \geq 2$, we require that $\langle \lng{s},\lng{t} \rangle^{k}$ is a reduced expression. 
		\end{enumerate}
\end{enumerate}
Given a Coxeter partition of $(\hat{W}, \hat{S})$, let $(W,S)$ be the Coxeter system with simple reflections $S$, where $st$ has order $m_{st}$ for each $s \ne t \in S$.
\end{defn}
We note here that condition 2(a) implies that $\langle \lng{s},\lng{t} \rangle^k$ is reduced for each $2 \leq k \leq m_{st}$, being a subword of the reduced expression $\langle \lng{s},\lng{t} \rangle^{m_{st}}$.

\begin{defn} The \emph{Coxeter embedding} (associated to the Coxeter partition $\pi$) is the homomorphism $\phi \co W \to \hat{W}$ defined on generators by $s \mapsto \phi(s) := \lng{s}$.
\end{defn}

It is straightforward to verify that the Coxeter embedding is well-defined.

Our first result is as follows:
\begin{thm} \label{strongermainthm} 
Coxeter embeddings are injective, and send reduced expressions to reduced expressions.
Moreover, for all $w \in W$ and $\hat{s} \in \hat{S}$, setting $s = \pi(\hat{s})$, we have $\phi(w) \hat{s} < \phi(w)$ if and only if $w s < w$. 
\end{thm}
This theorem is originally due to M\"{u}hlherr \cite{Muhlherr_CoxinCox}, see \S\ref{sec:relatedworks}; we provide a short proof of the theorem below. 
Examples of Coxeter partitions can be found in \S \ref{sec:examples}, and a special family of Coxeter partitions satisfying $\hat{W}_{U(s)} \cong \prod_i \ZZ/2\ZZ$ (i.e.\ a product of type $A_1$ finite Coxeter groups) for all $s \in S$ will be studied in \S \ref{sec:lusztigpart}.

We will use some standard notions and results for Coxeter groups, which can all be found (or deduced) from \cite[\S2, \S3 and \S4]{Comb_Coxeter_book}.
From here on, the partial order $\leq$ (and $<$) on any Coxeter system refers to the weak right Bruhat order.
We use the following notation. For elements $x, y, z$ in a Coxeter group, we write $x = y.z$ whenever $x = yz$ and $\ell(x) = \ell(y) + \ell(z)$; this is called a \emph{reduced
composition}. We write $\RD(w)$ for the \emph{right descent set} of $w$, i.e.\ the set of simple reflections $s \in S$ such that $ws<w$. Recall that $\RD(w)$ always generates a {\bf finite} Coxeter group. 
Moreover, if $x = y.z$ then $\RD(z) \subset \RD(x)$.

If $(W,S)$ is a Coxeter system and $J \subseteq S$, write $W^J$ for the set of minimal length representatives for cosets in $W/W_J$. Note that $W^J$ 
is the set of $w \in W$ such that $\RD(w) \cap J = \emptyset$.
If $y \in W^J$ and $z \in W_J$ then $y.z$ is reduced. Any $x \in W$ has a unique decomposition $x = y.z$ with $y \in W^J$ and $z \in W_J$.

\begin{lem}\label{lem:rightdescent}
Let $(W,S)$ be any Coxeter system and $J \subseteq S$.
Suppose we have $a \in J$, $B\in W_J$ and $C \in W$ such that $C.B$ is reduced but $CBa$ is not.
Then either $a \in \RD(B)$ or $\RD(C) \cap J \neq \emptyset$.
\end{lem}
\begin{proof}
Suppose $\RD(C) \cap J = \emptyset$. This implies $C \in W^J$.
Since $Ba \in W_J$, $C.(Ba)$ is reduced. If $Ba$ were reduced, then $CBa$ would be reduced. Thus $Ba$ is not reduced, implying $a \in \RD(B)$.
\end{proof}

\begin{proof}[Proof of Theorem \ref{strongermainthm}]
We show that a Coxeter embedding $\phi$ sends reduced expressions of $w \in W$ to reduced expressions in $\hat{W}$, by induction on the length of $w$. 
For $\ell(w) \leq 1$, the statement follows from the definition of Coxeter embeddings.
From now on let $w \in W$ with $\ell(w) \geq 2$, and assume via induction that the statement holds for all reduced expressions of elements in $W$ with length $< \ell(w)$. 

Pick a reduced expression for $w$, and let $s$ and $t$ be the last two letters of this reduced expression. Let $J = \{s,t\} \subset S$, and consider the unique decomposition $w = x.y$ with $x \in W^J$ and $y \in W_J$. Note that $\ell(y) \geq 2$. As such, $\ell(x) \le \ell(w)-2$ and so $\phi(x)$ is reduced by the inductive hypothesis. Moreover, as an element of the dihedral group $W_J$, $y$ is an alternating product of $s$ and $t$, so Condition (2) of Coxeter partitions guarantees that $\phi(y)$ is reduced. The statement is proven if $\phi(x) . \phi(y)$ is reduced.

Assume to the contrary that $\phi(x)\phi(y)$ is not reduced.
We decompose the reduced expression $\phi(y) \in \hat{W}_{\ufd(s) \sqcup \ufd(t)}$ into $B.a.A$ with $a \in \ufd(s) \sqcup \ufd(t) \subseteq \hat{S}$ so that $\phi(x).B$ is reduced but $\phi(x)Ba$ is not.
Note that $B.a$ is reduced by construction, so $a \not\in \RD(B)$.
By Lemma \ref{lem:rightdescent}, there exists some $b \in \RD(\phi(x)) \cap (\ufd(s) \sqcup \ufd(t)) \neq \emptyset$.
Suppose $b \in \ufd(s)$. Then $b \leq \lng{s} = \phi(s)$ and so $\phi(xs) = \phi(x)\phi(s)$ is not reduced.
But $xs=x.s$ is reduced (since $x \in W^{s,t}$) and $\ell(xs) = \ell(x) + 1 < \ell(w)$. By the induction hypothesis $\phi(x)\phi(s)$ must be reduced, which is a contradiction. The case $b \in \ufd(t)$ is treated in exactly the same fashion.

Since reduced expressions are sent to reduced expressions, $\phi$ has trivial kernel.

We now prove the final statement of the theorem. An equivalent statement is that either $s \in \RD(w)$ and $\ufd(s) \subset \RD(\phi(w))$, or $s \notin \RD(w)$ and $\ufd(s) \cap \RD(\phi(w)) = \emptyset$. Suppose $ws > w$. Then $\phi(w) \phi(s)$ is reduced so $\RD(\phi(w)) \cap \ufd(s) = \emptyset$. Conversely, suppose that $ws < w$. Then $w$ has a reduced expression ending in $s$, so $\phi(w)$ has a reduced expression ending in $\phi(s)$. Thus $\ufd(s) \subset \RD(\phi(w))$. 
\end{proof}

\section{Lusztig's partitions} \label{sec:lusztigpart}

We use the convention that an irreducible {\bf infinite} type Coxeter system $(W,S)$ has Coxeter number $\infty$.

\begin{defn} \label{defn:lusztig}
Let $(\hat{W}, \hat{S})$ and $(W,S)$ be two Coxeter systems, with $\hat{S}$ and $S$ finite sets.
A \emph{Lusztig's partition} is a surjective function $\pi \co \hat{S} \to S$ such that 
\begin{enumerate}
	\item For each $s \in S$, the elements of $\ufd(s)$ all commute with each other. \label{item:commute}
	\item For each pair $s \neq t \in S$, the parabolic subgroup $\hat{W}_{\ufd(s) \sqcup \ufd(t)}$ generated by $\ufd(s) \sqcup \ufd(t)$ is a product of Coxeter systems which all have the same Coxeter number $m_{st}$ (whether finite or infinite). 
\end{enumerate}
\end{defn}
Said another way, a Lusztig's partition is a coloring of the vertices $\hat{S}$ in the Coxeter graph of $\hat{W}$, such that no color is adjacent to itself, and such that any two colors form a disjoint union of Coxeter graphs with the same Coxeter number. 

It is well-known when $\hat{W}$ is finite that the morphism $\phi: W \to \hat{W}$ associated to a Lusztig's partition defined by
\[
\phi(s) = \prod_{\hat{s} \in \ufd(s)} \hat{s}
\]
sends reduced expressions to reduced expressions (e.g.\ via Lemma \ref{lem:Bourbaki} below).
We prove that this remains true even when $\hat{W}$ is infinite, by showing that Lusztig's partitions are a specific type of Coxeter partitions.
In fact, we will also show that Lusztig's partitions are precisely Coxeter partitions which satisfy (the stronger) condition \eqref{item:commute} of Definition \ref{defn:lusztig}.
In other words, Lusztig's partitions classify the ``simplest'' type of Coxeter partitions, where the elements of $U(s)$ are all pair-wise disjoint in the Coxeter graph of $(\hat{W}, \hat{S})$.
These will all be the content of Theorem \ref{prop:thmfordihedral}.

We begin by recalling two lemmas about reduced expressions for powers of Coxeter elements.  The first lemma is classical, and the second is a result of Speyer.

\begin{lem}[\protect{\cite[Chapter V \S6 Exercise 2]{Bourbaki_Lie_4-6}}] \label{lem:Bourbaki}
Let $(W,S)$ be an irreducible finite Coxeter system with Coxeter number $h$.
Let $S= S' \sqcup S''$ be a partition of $S$ so that in each subset all elements commute.
Denote $x:= \prod_{s'\in S'} s'$ and $y:= \prod_{s'' \in S''} s''$.
Then the longest element $w_0$ is given by 
\[
w_0 = \langle y, x \rangle^h = \langle x, y \rangle^h
\]
and both expressions are reduced.
In particular, the largest power of the Coxeter element $c:= x.y$ that is reduced is $c^{\lfloor h/2 \rfloor}$.
\end{lem}

\begin{lem}[\protect{\cite[Theorem 1]{SpeyerReduced}}] \label{lem:Speyer}
Let $(W,S)$ be an irreducible infinite Coxeter system. For any choice of Coxeter element $c \in W$, $c^k$ is a reduced expression for all $k$.
\end{lem}

Our second result is the following theorem, originally due to Dyer \cite{Dyer_rootII}, see \S\ref{sec:relatedworks}.

\begin{thm} \label{prop:thmfordihedral}
Lusztig's partitions are Coxeter partitions. Conversely, a Coxeter partition where the elements of $U(s)$ all commute with each other (for each $s \in S$) is a Lusztig's partition.
\end{thm}

\begin{proof} 

Given a surjection $\pi: \hat{S} \rightarrow S$ satisfying condition (1) of a Lusztig's partition, it is immediate that condition (1) of a Coxeter partition is satisfied. It suffices to prove that condition (2) of a Lusztig's partition is equivalent to condition (2) of a Coxeter partition, in this case. Condition (2) is a ``rank two condition,'' namely it is a condition on each pair of distinct vertices $s \neq t \in S$.
As such, it suffices to prove the case where $S = \{s,t\}$ is a two-element set.
This will be the case considered in the rest of this proof.

First suppose that $(\hat{W},\hat{S})$ corresponds to a {\bf connected} Coxeter graph. Then $\hat{S}$ has a unique bipartite coloring (up to swapping the colors). Condition (2) of a Lusztig's partition is now vacuous, so we must prove that condition (2) of a Coxeter partition holds. If $\hat{W}$ is finite with Coxeter number $m_{st}$, that $\langle \lng{s}, \lng{t} \rangle^{m_{st}}$ is a reduced expression for the longest element of $\hat{W}$ follows from Lemma \ref{lem:Bourbaki}.
If $\hat{W}$ is instead infinite, consider the Coxeter element $\hat{c}$ of $\hat{W}$ given by
\[
\hat{c} := \lng{s}\lng{t} = 
	\left( \prod_{\hat{s} \in \ufd(s)} \hat{s} \right)\left( \prod_{\hat{t} \in \ufd(t)} \hat{t} \right).
\]
By Lemma \ref{lem:Speyer}, any power of $\hat{c}$ is reduced\footnote{In fact, we only need a weaker version of Lemma \ref{lem:Speyer}, which was proven earlier in \cite[Corollary 9.6]{FZ_clustercoeff}.}. 
But for all $k \geq 2$, $\langle \lng{s},\lng{t} \rangle^{k}$ is always a subword of some power of $\hat{c}$, so $\langle \lng{s},\lng{t} \rangle^{k}$ must itself be reduced.

Now consider the general case: $\hat{S}$ is a disjoint union of connected Coxeter graphs $\hat{S}_i$. Then $\hat{W}$ is a product $\prod_i \hat{W}_i$. 
For each $i$, we use the shorthand
\[ 
\lngi{s} := \prod_{\hat{s} \in \ufd(s) \cap \hat{S}_i} \hat{s},
\]
so that $\lng{s}$ is the product of all $\lngi{s}$ (in any order).
For all $k \geq 2$, the following equality (in $\hat{W}$) can be obtained by only applying commutativity relations:
\begin{align*}
\langle \lng{s},\lng{t} \rangle^k
	&= \left\langle \prod_i \lngi{s}, \prod_i \lngi{t} \right\rangle^k \\
	&= \prod_i \langle \lngi{s}, \lngi{t} \rangle^k.
\end{align*}

If each $\hat{W}_i$ has the same Coxeter number $m_{st}$ (finite or infinite), then by the connected case above, $\langle \lngi{s}, \lngi{t} \rangle^k$ is reduced for $2 \le m \le m_{st}$. As a (direct) product of reduced expressions, $\langle \lng{s}, \lng{t} \rangle^k$ is also reduced, which is condition (2) of a Coxeter partition.

Conversely, let $m_{st}$ be the order of $\lng{s}\lng{t}$ (finite or infinite). If some $\hat{W_i}$ has Coxeter number $h_i$ strictly less than $m_{st}$, then for any $h_i < k \le m_{st}$ we have that $\langle \lngi{s}, \lngi{t} \rangle^k$ is not a reduced expression. Consequently, neither is $\langle \lng{s}, \lng{t} \rangle^k$. So if condition (2) of a Coxeter partition holds, then $h_i \ge m_{st}$ whenever $m_{st}$ is finite, or $h_i = \infty$ whenever $m_{st}$ is infinite. Since $h_i$ divides $m_{st}$ whenever both are finite (in order for $(\lng{s}\lng{t})^{m_{st}}$ to be the identity), we deduce that $h_i = m_{st}$. This proves condition (2) of a Lusztig's partition.
\end{proof}

Condition (2) of Lusztig's partitions is a ``same Coxeter number'' property for connected components of  subgraphs generated by $\ufd(s) \sqcup \ufd(t)$. We now show that the ``same Coxeter number'' property extends to connected components of the whole Coxeter graph\footnote{
By restricting the Coxeter partition from $\hat{S}$ to a suitable subset, we also obtain the result for $\ufd(s) \sqcup \ufd(t) \sqcup \ufd(u)$ and other preimages.
} 
of $(\hat{W}, \hat{S})$.
Note that this is not true for Coxeter partitions in general; see Example \ref{ex:diffcoxnum}.

\begin{prop} \label{prop:samecoxnum}
Let $\pi: \hat{S} \rightarrow S$ be a Lusztig's partition and let $(W,S)$ be an irreducible Coxeter system with Coxeter number $h$.
Then the irreducible components of $(\hat{W}, \hat{S})$ also have the same Coxeter number $h$.
\end{prop}
\begin{proof}
Throughout, we will implicitly use the fact that a Lusztig's partition is a Coxeter partition, shown in Theorem \ref{prop:thmfordihedral}.

Since $\pi$ is a Lusztig's partition, the associated Coxeter embedding $\phi:W \rightarrow \hat{W}$ sends any Coxeter element $c\in W$ to a (mutually commuting) product $\prod_i \hat{c}_i$ of Coxeter elements $\hat{c}_i$ for each irreducible component $(\hat{W}_i, \hat{S}_i)$ of $(\hat{W}, \hat{S})$.
By Theorem \ref{strongermainthm}, $\phi$ sends reduced expressions to reduced expressions, hence if $c^k$ is reduced, then so is $\prod_i (\hat{c}_i)^k$.
Since the $\hat{c}_i$'s mutually commute,  the final statement is equivalent to each $(\hat{c}_i)^k$ being reduced.

Suppose $h = \infty$. By Lemma \ref{lem:Speyer}, $c^k$ is reduced for all $k > 0$, and thus so is $(\hat{c}_i)^k$. This shows that each irreducible component $(\hat{W}_i, \hat{S}_i)$ is infinite type, as desired.

Now let us assume $h < \infty$, and prove that $h_i = h$, where $h_i$ is the Coxeter number of the component $(\hat{W}_i, \hat{S}_i)$.
Take the Coxeter element $c = x.y \in W$ associated to some (two-coloring) partition of $S$ as in Lemma \ref{lem:Bourbaki}, so that $c^{\lfloor h/2 \rfloor}$ is a reduced expression.
By Theorem \ref{strongermainthm}, we get that
\[
\phi(c^{\lfloor h/2 \rfloor}) = \prod_i (\hat{c}_i)^{\lfloor h/2 \rfloor}
\]
is a reduced expression, and thus each $(\hat{c}_i)^{\lfloor h/2 \rfloor}$ is a reduced expression.
In particular, the order $h_i$ of $\hat{c}_i$ satisfies $h_i >\lfloor h/2 \rfloor$.
Since $(\hat{c}_i)^h = \id$, $h_i$ divides $h$.
A simple numerical argument shows that $h_i = h$ as required. 
\end{proof}

\section{Examples}\label{sec:examples}
Coxeter embeddings associated to Lusztig's partitions (when each $m_{st}$ is finite) were studied by Lusztig in \cite[\S 3]{LusSqint}. These include examples coming from folding by graph symmetries, such as the first example below. 
Lusztig's famous inclusion of Coxeter groups of type $H_4$ into type $E_8$ is a Lusztig's partition that does not come from graph symmetries. 
Our setting of Lusztig's partitions includes examples where $m_{st}$ is infinite (see Example \ref{ex:infinitelabel}). 

We present below some examples of Lusztig's partitions, followed by some examples of Coxeter partitions which are not Lusztig's partitions.
In all examples to follow, the elements in $\hat{S}$ are given by alphabets with subscripts and elements in $S$ are given by alphabets without subscripts.
The partition map $\pi: \hat{S} \rightarrow S$ is defined by forgetting the subscript.

\begin{ex} \label{ex:D4G2}
Consider the folding of $D_4$ onto $G_2$. 
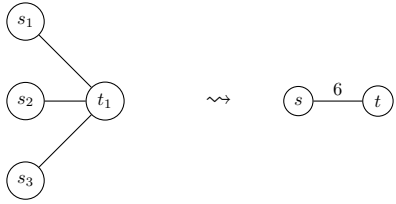
\begin{center}
\begin{tikzpicture}[node distance={10ex}, main/.style = {draw, circle, scale=0.7}, label/.style = {midway, above, scale=0.7}] 
\node[main] (s1) {$s_1$}; 
\node[main] (s2) [below of = s1] {$s_2$};	
	\node[main] (t1) [right of = s2] {$t_1$};
\node[main] (s3) [below of = s2] {$s_3$};

\draw (s1) -- (t1);
\draw (s2) -- (t1);
\draw (s3) -- (t1);

\node[] (arrow) [right of = t1] {$\rightsquigarrow$};

\node[main] (s) [right of = arrow] {$s$};
\node[main] (t) [right of = s] {$t$};
\draw (s) -- node[label]{$6$} (t);
\end{tikzpicture}
\end{center}
This is an example of folding from a graph symmetry, because the map $\pi$ records the orbits under a group action on the Coxeter graph. 
\end{ex}

\begin{ex} \label{ex:otherG2}
Here are two different Lusztig's partitions, giving embeddings of $G_2$ which do not come from graph symmetries. 
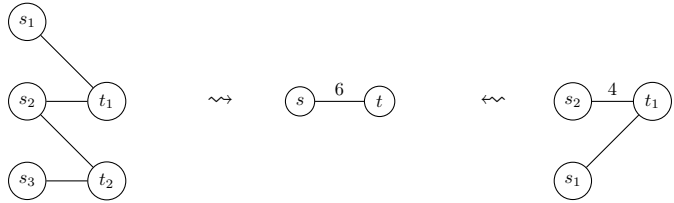
\begin{center}
\begin{tikzpicture}[node distance={10ex}, main/.style = {draw, circle, scale=0.7}, label/.style = {midway, above, scale=0.7}] 
\node[main] (s1) {$s_1$}; 
\node[main] (s2) [below of = s1] {$s_2$};	
	\node[main] (t1) [right of = s2] {$t_1$};
\node[main] (s3) [below of = s2] {$s_3$};
	\node[main] (t2) [right of = s3] {$t_2$};

\draw (s1) -- (t1);
\draw (s2) -- (t1);
\draw (s2) -- (t2);
\draw (s3) -- (t2);

\node[] (arrow) [right of = t1] {$\rightsquigarrow$};

\node[main] (s) [right of = arrow] {$s$};
\node[main] (t) [right of = s] {$t$};
\draw (s) -- node[label]{$6$} (t);

\node[] (arrow2) [right of = t] {\reflectbox{$\rightsquigarrow$}};
 
\node[main] (s2') [right of = arrow2] {$s_2$}; 
\node[main] (t1') [right of = s2'] {$t_1$};
\node[main] (s1') [below of = s2'] {$s_1$};

\draw (s2') -- node[label]{$4$} (t1');
\draw (s1') -- (t1');

\end{tikzpicture}
\end{center}
\end{ex}

\begin{ex} \label{ex:infinitelabel}
Let $\hat{\Gamma}$ be a bipartite Coxeter graph, all of whose connected components have the same Coxeter number $h$ (possibly infinite). Any two coloring of $\hat{\Gamma}$ defines a Lusztig's partition onto the dihedral group $I_2(h)$. 
Below is an example for $h=\infty$.

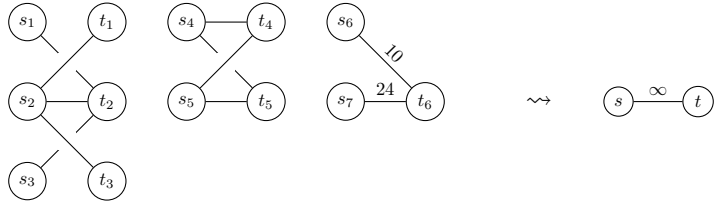
\begin{center}
\begin{tikzpicture}[node distance={10ex}, main/.style = {draw, circle, scale=0.7}, label/.style = {midway, above, scale=0.7}] 
\node[main] (s1) {$s_1$}; 
\node[main] (s2) [below of = s1] {$s_2$};	
\node[main] (s3) [below of = s2] {$s_3$};

\node[main] (t1) [right of = s1] {$t_1$};
\node[main] (t2) [right of = s2] {$t_2$};
\node[main] (t3) [right of = s3] {$t_3$};

\node[main] (s4) [right of = t1] {$s_4$};
\node[main] (s5) [right of = t2] {$s_5$};
\node[main] (t4) [right of = s4] {$t_4$};
\node[main] (t5) [right of = s5] {$t_5$};

\node[main] (s6) [right of = t4] {$s_6$};
\node[main] (s7) [right of = t5] {$s_7$};
\node[main] (t6) [right of = s7] {$t_6$};

\draw (s1) -- node[fill,white]{} (t2); 
\draw (s2) -- (t1);
\draw (s3) -- node[fill,white]{} (t2);
\draw (s2) -- (t2);
\draw (s2) -- (t3);

\draw (s4) -- (t4); 
\draw (s4) -- node[fill,white]{} (t5);
\draw (s5) -- (t4);
\draw (s5) -- (t5);

\draw (s6) -- node[label,sloped]{$10$} (t6);
\draw (s7) -- node[label]{$24$} (t6);

\node[] (arrow) [right of = t6] {$\rightsquigarrow$};

\node[main] (s) [right of = arrow] {$s$};
\node[main] (t) [right of = s] {$t$};
\draw (s) -- node[label]{$\infty$} (t);
\end{tikzpicture}
\end{center}
\end{ex}

\begin{ex} 
The following is an example of a Lusztig's partition similar to the Lusztig's partition from $E_8$ onto $H_4$, except with type $B$ and type $D$ tails.
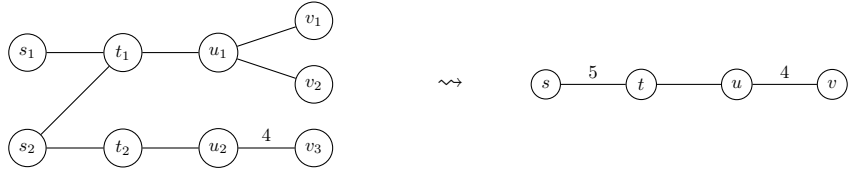
\begin{center}
\begin{tikzpicture}[node distance={12ex}, main/.style = {draw, circle, scale=0.7}, label/.style = {midway, above, scale=0.7}] 
\node[main] (s1) {$s_1$}; 
\node[main] (s2) [below of = s1] {$s_2$};

\node[main] (t1) [right of = s1] {$t_1$};	
\node[main] (t2) [right of = s2] {$t_2$};
	
\node[main] (u1) [right of = t1] {$u_1$};
\node[main] (u2) [right of = t2] {$u_2$};

\node[main] (v1) [right of = u1, yshift=4ex] {$v_1$};
\node[main] (v2) [right of = u1, yshift=-4ex] {$v_2$};
\node[main] (v3) [right of = u2] {$v_3$};

\draw (s1) -- (t1);
\draw (s2) -- (t1);
\draw (s2) -- (t2);

\draw (t1) -- (u1);
\draw (t2) -- (u2);

\draw (u1) -- (v1);
\draw (u1) -- (v2);
\draw (u2) -- node[label]{$4$} (v3);

\node[] (arrow) [right of = v2] {$\rightsquigarrow$};

\node[main] (s) [right of = arrow] {$s$};
\node[main] (t) [right of = s] {$t$};
\node[main] (u) [right of = t] {$u$};
\node[main] (v) [right of = u] {$v$};

\draw (s) -- node[label]{$5$} (t);
\draw (t) -- (u);
\draw (u) -- node[label]{$4$} (v);
\end{tikzpicture}
\end{center}
\end{ex}

\begin{ex} This example takes an embedding of $I_2(12)$ into $E_6 \times I_2(12)$, and adds some fun. 
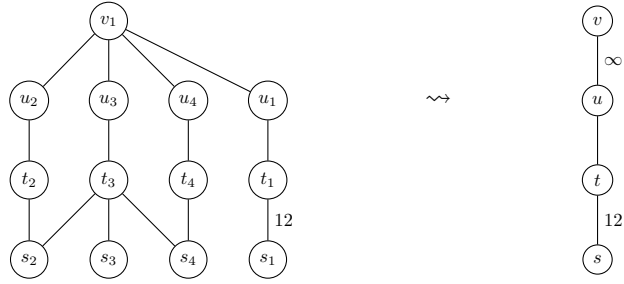
\begin{center}
\begin{tikzpicture}[node distance={10ex}, main/.style = {draw, circle, scale=0.7}, label/.style = {midway, right, scale=0.7}] 
\node[main] (v1) {$v_1$};
\node[main] (u3) [below of = v1] {$u_3$};
\node[main] (u2) [left of  = u3] {$u_2$};
\node[main] (u4) [right of = u3] {$u_4$};
\node[main] (u1) [right of = u4] {$u_1$};
\node[main] (t3) [below of = u3] {$t_3$};
\node[main] (s3) [below of = t3] {$s_3$};
\node[main] (t2) [below of = u2] {$t_2$};
\node[main] (t4) [below of = u4] {$t_4$};
\node[main] (s2) [below of = t2] {$s_2$};
\node[main] (s4) [below of = t4] {$s_4$};
\node[main] (t1) [below of = u1] {$t_1$};	
\node[main] (s1) [below of = t1] {$s_1$};

\draw (s1) -- node[label]{$12$} (t1);
\draw (t1) -- (u1);
\draw (u1) -- (v1);
\draw (v1) -- (u2);
\draw (v1) -- (u3);
\draw (v1) -- (u4);
\draw (u2) -- (t2);
\draw (u3) -- (t3);
\draw (u4) -- (t4);

\draw (t2) -- (s2);
\draw (s2) -- (t3);
\draw (t3) -- (s4);
\draw (t3) -- (s3);
\draw (s4) -- (t4);

\node[] (arrow) [right of = u1, xshift=5ex] {$\rightsquigarrow$};

\node[main] (u) [right of = arrow, xshift=10ex] {$u$};
\node[main] (v) [above of = u] {$v$};
\node[main] (t) [below of = u] {$t$};
\node[main] (s) [below of = t] {$s$};

\draw (s) -- node[label]{$12$} (t);
\draw (t) -- (u);
\draw (u) -- node[label]{$\infty$} (v);
\end{tikzpicture}
\end{center}

\end{ex}

We conclude with some examples of Coxeter partitions that are not Lusztig's partitions.

\begin{ex}
Consider the following partition of $A_4$ onto $B_2$:
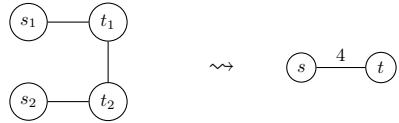
\begin{center}
\begin{tikzpicture}[node distance={10ex}, main/.style = {draw, circle, scale=0.7}, label/.style = {midway, above, scale=0.7}] 
\node[main] (s1) {$s_1$}; 
\node[main] (t1) [right of = s1] {$t_1$};	
\node[main] (t2) [below of = t1] {$t_2$};
\node[main] (s2) [below of = s1] {$s_2$};

\draw (s1) -- (t1);
\draw (t1) -- (t2);
\draw (t2) -- (s2);

\node[] (arrow) [right of = t1, yshift=-4ex] {$\rightsquigarrow$};

\node[main] (s) [right of = arrow] {$s$};
\node[main] (t) [right of = s] {$t$};
\draw (s) -- node[label]{$4$} (t);
\end{tikzpicture}
\end{center}
This is not a Lusztig's partition since $\ufd(t) = \{t_1,t_2\}$ and $t_1t_2 \neq t_2t_1$. (Note that $A_4$ and $B_2$ have different Coxeter numbers; cf.\ Proposition \ref{prop:samecoxnum}.)
Nonetheless, $\ufd(s)$ and $\ufd(t)$ both generate a finite parabolic subgroup, so condition (1) of a Coxeter partition holds.
The corresponding longest elements are $w_0^s = s_1s_2 = s_2s_1$ and $w_0^t = t_1t_2t_1 = t_2t_1t_2$ respectively.
A straightforward calculation shows that 
\[
\langle w_0^s, w_0^t \rangle^4 = s_2s_1t_2t_1t_2s_2s_1t_2t_1t_2 
\]
is a reduced expression for the longest element of $A_4$. The order of $w_0^sw_0^t$ is indeed $4$, since its square is the longest element and has order $2$.
This shows that condition (2) is satisfied.
Similarly, we have a Coxeter partition (that is not a Lusztig's partition) of affine type $D_5$ onto affine type $C_2$:
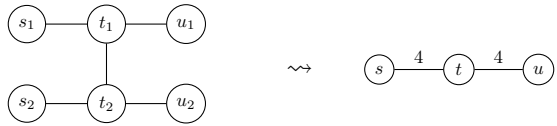
\begin{center}
\begin{tikzpicture}[node distance={10ex}, main/.style = {draw, circle, scale=0.7}, label/.style = {midway, above, scale=0.7}] 
\node[main] (s1) {$s_1$}; 
\node[main] (s2) [below of = s1] {$s_2$};
\node[main] (t1) [right of = s1] {$t_1$};	
\node[main] (t2) [below of = t1] {$t_2$};
\node[main] (u1) [right of = t1] {$u_1$};	
\node[main] (u2) [below of = u1] {$u_2$};

\draw (s1) -- (t1);
\draw (t1) -- (t2);
\draw (t2) -- (s2);
\draw (t1) -- (u1);
\draw (t2) -- (u2);

\node[] (arrow) [right of = u1, yshift=-4ex] {$\rightsquigarrow$};

\node[main] (s) [right of = arrow] {$s$};
\node[main] (t) [right of = s] {$t$};
\node[main] (u) [right of = t] {$u$};
\draw (s) -- node[label]{$4$} (t);
\draw (t) -- node[label]{$4$} (u);
\end{tikzpicture}
\end{center}
The reader is encouraged to construct the affine-type examples given in \cite[Table 2]{CrispInjective}.
\end{ex}

\begin{ex}\label{ex:diffcoxnum}
Note that Coxeter partitions (but not Lusztig's partitions; see Proposition \ref{prop:samecoxnum}) allow mixed Coxeter numbers, as the following example from $A_3 \sqcup A_4$ to $B_2$ shows:
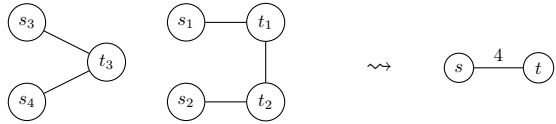
\begin{center}
\begin{tikzpicture}[node distance={10ex}, main/.style = {draw, circle, scale=0.7}, label/.style = {midway, above, scale=0.7}] 
\node[main] (s1) {$s_1$}; 
\node[main] (t2) [right of = s1] {$t_1$};	
\node[main] (t3) [below of = t2] {$t_2$};
\node[main] (s4) [below of = s1] {$s_2$};

\node[main] (t1) [left of = s1, yshift=-5ex] {$t_3$};

\node[main] (s2) [left of = t1, yshift=5ex] {$s_3$};
\node[main] (s3) [below of = s2] {$s_4$};

\draw (s1) -- (t2);
\draw (t2) -- (t3);
\draw (t3) -- (s4);

\draw (s2) -- (t1);
\draw (s3) -- (t1);

\node[] (arrow) [right of = t2, yshift=-4ex] {$\rightsquigarrow$};

\node[main] (s) [right of = arrow] {$s$};
\node[main] (t) [right of = s] {$t$};
\draw (s) -- node[label]{$4$} (t);
\end{tikzpicture}
\end{center}
\end{ex}

\begin{rem}
Given a Coxeter system $(\hat{W},\hat{S})$ with $s_+,s_- \in \hat{S}$ such that $\hat{m}_{s_+ s_-} = 3$, an \emph{edge contraction} \cite[\S 1.3]{li_heckecontract} can be performed (on the Coxeter graph) to obtain a new Coxeter system $(W,S)$ defined by
\begin{align*}
S &:= \left(\hat{S} \cup \{s_0\} \right)\setminus \{s_+,s_-\}; \\
m_{st} &:= \begin{cases}
	\hat{m}_{st}, &\text{ if } s,t \neq s_0; \\
	\hat{m}_{ss_+} + \hat{m}_{ss_-} - 2, &\text{ if } t=s_0, \hat{m}_{s s_+} = 2 \text{ or } \hat{m}_{s s_-} = 2; \\
	\infty, &\text{ if } t=s_0, \hat{m}_{s s_+}, \hat{m}_{s s_-} > 2.
	\end{cases}
\end{align*}
The group homomorphism $\phi: W \rightarrow \hat{W}$ sending $s_0 \mapsto s_+s_-s_+$ with the identity on all other generators was shown to be injective in \cite[Proposition 1.4.1]{li_heckecontract}.
While $\phi$ feels similar to a ``Coxeter embedding'', we note that the natural surjection $\pi: \hat{S} \rightarrow S$ defined by $s_+, s_- \mapsto s_0$ and identity on all other generators is \emph{not} a Coxeter partition (condition (2) need not be satisfied).
In fact, even though $\phi$ is injective, it need not send reduced expressions to reduced expressions (while a Coxeter embedding always does); this can already be seen from an edge contraction on the $A_3$ Coxeter graph.
\end{rem}

\section{Related works}\label{sec:relatedworks}

We discovered after this paper was originally posted that none of our results (nor perhaps the proofs) are novel. We chose to keep our exposition and proofs intact above, and discuss this prior work here.

\subsection*{Admissible partitions; after M\"{u}hlherr}
Our definition of Coxeter partitions turns out to be equivalent to M\"{u}hlherr's definition of \emph{admissible partitions}; see \cite[pg.\ 278]{Muhlherr_CoxinCox} for the definition and \cite[Theorem 1.2 and Lemma  3.3]{Muhlherr_CoxinCox} for the fact that they are equivalent (see also \cite[Lemma 14]{Castella_admissiblesubmonoid}).

The injective part of our Theorem \ref{strongermainthm} is \cite[Theorem 1.1]{Muhlherr_CoxinCox}, whereas the preservation of reduced expressions is \cite[Proposition 4.5 A1]{Muhlherr_CoxinCox}. 
The final statement of Theorem \ref{strongermainthm} is in fact equivalent to Condition (A) of admissible partitions. 
Our proof of Theorem \ref{strongermainthm} above is different from (and shorter than) M\"{u}hlherr's proof, and we now give context on that proof.

\subsection*{LCM-homomorphisms and Artin (sub)monoids}
Oblivious to the earlier work by M\"{u}hlherr, our original intention was to use the results of LCM-homomorphisms (between Artin monoids) from Crisp's work \cite{CrispInjective} to obtain a weaker version of Theorem \ref{strongermainthm}; cf.\ \cite[Proposition 2.3]{CrispInjective}.
In the process, we also realized that the definition of LCM-homomorphisms and the proof of \cite[Lemma 2.2]{CrispInjective} (or similarly, \cite[Proposition 2.6]{godelle_injectif}) could be modified to obtain a direct proof of Theorem \ref{strongermainthm} in the Coxeter group setting, skipping the theory of Artin monoids.

The modification\footnote{Condition 2(a) of Coxeter partitions is essentially the definition of LCM-homomorphisms in the free-of-infinity setting considered by Crisp in \cite{CrispInjective}, and Condition 2(b) was what we added to include the free-of-infinity setting. On the other hand, Godelle's condition \textbf{(L3)} in his version of LCM-homomorphisms \cite[Definition 2.1]{godelle_injectif} for the general setting is strictly stronger than condition 2(b).} required led us to the definition of Coxeter partitions, which we later found out is equivalent to M\"{u}hlherr's definition of admissible partitions.
We later found out that the generalization of both the works \cite{CrispInjective, godelle_injectif} to the setting of admissible partitions (for Artin monoids) was studied by Castella in \cite{Castella_admissiblesubmonoid}.
Indeed, in Remark 35 of loc.\ cit.\ Castella suggests that the proofs of Crisp and Godelle for LCM-homomorphisms would also work in the more general setting of admissible partitions. Unwittingly, our proof of Theorem \ref{strongermainthm} is the realization of this remark, having stripped away the theory of Artin monoids.

\subsection*{Dyer's work}
Meanwhile, Dyer in \cite[Section 8]{Dyer_rootII} proves a generalization of Thoerem \ref{strongermainthm} (citing M\"{u}hlherr), and also proves Thoerem \ref{prop:thmfordihedral} in this generalized context. His proof of Theorem \ref{prop:thmfordihedral} is extremely similar to ours, including the use of Speyer's result (Lemma \ref{lem:Speyer}). Dyer too seems to have been unaware of prior work of Crisp or Castella or Godelle.

\subsection*{Acknowledgements} 
The authors would like to thank Federica Gavazzi for pointing out the work of Godelle, which led us down the rabbit hole to the work of Castella then M\"{u}hlherr and finally Dyer.
We also thank Yiqiang Li for informing us of his interesting work.
B.E. was supported by NSF grant DMS-2201387, and his research group was supported by DMS-2039316.

\printbibliography

\end{document}